\documentclass{amsart}

\usepackage{enumitem}

\theoremstyle{definition}
\newtheorem{theorem}{Theorem}
\newtheorem{lemma}[theorem]{Lemma}

\newtheorem{remark}[theorem]{Remark}

\allowdisplaybreaks
\newcommand{\ip}[2]{\langle #1,#2\rangle}
\newcommand{\abs}[1]{|#1|}
\numberwithin{equation}{section}
\renewcommand{\Pi}{\overline{\pi}}

\title[Uncertainty principles and differential operators]{Uncertainty principles and differential operators
  on the weighted {B}ergman space}
\author{Jens Gerlach Christensen}
\address{ Department of Mathematics, Colgate University} 
\email{jchristensen@colgate.edu}
\urladdr{http://www.math.colgate.edu/~jchristensen}

\author{Christopher Benjamin Deng}
\email{cdeng@colgate.edu}

\begin{document}

\maketitle

\begin{abstract}
  We classify self-adjoint first-order differential operators
  on weighted Bergman spaces on the unit disc and answer questions related to
  uncertainty principles for such operators. 
  Our main tools are the discrete series representations
  of $\mathrm{SU}(1,1)$. This approach has the promise to generalize to
  other bounded symmetric domains.
\end{abstract}

\section{Introduction}
In this paper we address several questions about weighted Bergman spaces on the unit disc from the perspective of representation theory.
We rederive an uncertainty principle by Soltani \cite{Soltani2021} and give a partial answer to an open question posed by Yong and Zhu \cite{Zhu2015}. The question is if there are self-adjoint operators on the Bergman space
for which the commutator is a non-zero multiple of the identity. We show
that this cannot happen if the operators are first-order differential operators.
Also, we give a classification of first-order differential operators
on the weighted Bergman space.
Such a classification has already been carried out for unweighted Bergman space in \cite{Villone1973,Villone1970}.
Moreover, we show that they
arise from the derived representation of the discrete series representations.
The representation theoretic approach we propose seems to generalize to bounded symmetric domains.
Lastly, we set up an isomorphism between Bergman spaces with shifted weights in the Hilbert space case,
and conjecture that this extends to the whole family of Bergman spaces.

\section{Weighted Bergman Space}
Let $H(\mathbb{D})$ be the space of all holomorphic functions on the unit disk $\mathbb{D}=\{z\in\mathbb{C}:|z|<1 \} $.
Let $dz$ denote the Euclidean measure on $\mathbb{C}$ and
define the weighted measure
$d\nu_\xi(z):=\frac{(\xi+1)}{\pi}(1-|{z}|^2)^\xi \,dz$. 
For $\xi>-1$, this is a probability measure and therefore
the weighted Bergman space defined by
\begin{align}
  \mathcal{A}_\xi^p:=\left\{ f\in H(\mathbb{D})\,\bigg|\,\|f\|_{\mathcal{A}_\xi^p}:=\left(\int_{\mathbb{D}}|f(z)|^2\,d\nu_\xi(z)\right)^{1/p} <\infty\right\} \nonumber
\end{align}
is non-trivial for $\xi>-1$.
The weighted Bergman space is a reproducing kernel
Banach space for $1\leq p<\infty$, and the reproducing kernel given by
$K_\xi(z,w) = \frac{1}{(1-z\overline{w})^{\xi+2}}$
satisfies
$f(z) = \int_\mathbb{D} f(w)K_\xi(z,w)\,d\nu_\xi(w)$ for $f\in \mathcal{A}_\xi^p$.

In this paper we will mostly focus on the Hilbert space
$\mathcal{A}^2_\xi$ with inner product
$$
\ip{f}{g}_{\xi} := \int_{\mathbb{D}} f(z)\overline{g(z)}\, d\nu_\xi(z),
$$
which has an orthonormal basis $e_n(z):=\sqrt{\frac{(\xi+2)_n}{n!}}z^n$.
We also denote the norm on this space by $\|\cdot\|_\xi$.

We define additional Hilbert spaces of analytic functions over $\mathbb{D}$
	\begin{align}
		\mathcal{A}_{\xi,n}^2:=\left\{f\in H(\mathbb{D})\,\bigg|\,\|f\|^2_{\xi,n}:=|f(0)|^2+\int_{\mathbb{D}}|f^{(n)}(z)|^2d\nu_\xi(z)<\infty\right\} \nonumber
	\end{align}
	for $n\in\mathbb{N}$.
        These spaces will later be identified as domains of certain differential operators.
	Notice that the spaces $\mathcal{A}_{\xi,n}^2$ satisfies the containment $\cdots\subseteq\mathcal{A}^2_{\xi,n+1}\subseteq\mathcal{A}_{\xi,n}^2\subseteq\cdots\subseteq\mathcal{A}_{\xi,1}^2\subseteq\mathcal{A}_{\xi}^2$.

\section{Uncertainty principles and Lie groups}
Let $H$ be a Hilbert space with inner product denoted $\ip{\cdot}{\cdot}$
and norm $\|\cdot\|$. For an operator $A:H\to H$ the domain will be denoted
$\mathcal{D}(A)$. Assuming that $A$ is densely defined and closable, let
$\overline{A}$ denote the closure of $A$. For self- or skew-adjoint operators $A,B:H\longrightarrow H$, the following uncertainty principle can easily be derived
\begin{align}
  |{\ip{[A,B]u}{u}}|\leq 2\|(A+a)u\| \|(B+b)u\| ~ \text{for all} ~ u\in\mathcal{D}(A)\cap\mathcal{D}(B)\cap \mathcal{D}([A,B]) \nonumber
\end{align}
where $a,b\in\mathbb{R}$.
In general it is not true that this uncertainty principle
extends to
\begin{align*}
  |{\ip{\overline{[A,B]}u}{u}}|\leq 2\|(A+a)u\| \|(B+b)u\| ~ \text{for all} ~ u\in\mathcal{D}(A)\cap\mathcal{D}(B)\cap \mathcal{D}(\overline{[A,B]}),
\end{align*}
since $[A,B]$ may not even be densely defined.
However, when the operators involved arise from representations of Lie groups
as described below, the uncertainty principle can be extended.

Recall that $(\pi,H)$ is a representation of a group $G$ if it satisfies
\begin{enumerate}[label=\text{\alph*)}]
  \setlength{\itemsep}{0pt}
\item $\pi(x):H\longrightarrow H$ is linear,
\item $\pi(e)=I$,
\item $\pi(xy)=\pi(x)\pi(y)$
\end{enumerate}
for $x,y\in G$.
The representation is called unitary if every $\pi(x)$ is unitary.

Let $G$ be a Lie group with Lie algebra $\mathfrak{g}$ and
exponential mapping $\exp:\mathfrak{g}\to G$
(we will also denote $\exp(X)$ by $e^X$). The Lie bracket on $\mathfrak{g}$
will be denoted $[\cdot,\cdot]$.

If $(\pi,H)$ is a representation
of the Lie group $G$, then the space of smooth vectors $H_\pi^\infty$ is
the collection of vectors $v$ in $H$ for which
$x\mapsto \pi(x)v$ is smooth from $G$ to $H$. It is well-known that
$H_\pi^\infty$ is dense in $H$. Therefore
an $X\in\mathfrak{g}$ defines a densely defined
differential operator $\pi(X)$ by
\begin{equation}\label{diffop}
\pi(X)v = \frac{d}{dt}\Big|_{t=0} \pi(\exp(tX))v
\end{equation}
with domain $H_\pi^\infty$.
If $\pi$ is unitary, the operator $\pi(X)$ is skew-symmetric and thus closable, and its
closure, denoted $\Pi(X)$, is skew-adjoint. 
The domain of $\Pi(X)$ is exactly the collection of vectors
for which (\ref{diffop}) converges in $H$.
It turns out that for $X,Y\in\mathfrak{g}$
the operator $[\Pi(X),\Pi(Y)]$ is closable and
its closure is $\Pi([X,Y])$.
By \cite{Christensen2004} the following uncertainty principle holds.
\begin{theorem} \label{Lie UP}
  Let $G$ be a Lie group with Lie algebra $\mathfrak{g}$, and let $(H,\pi)$ be a unitary representation of $G$.
  Suppose that $X,Y\in\mathfrak{g}$.
  Then for $x,y\in\mathbb{R}$,
  \begin{align}
    |{\ip{\Pi([X,Y])u}{u}}|\leq2\|(\Pi(X)+x)u\|\|(\Pi(Y)+y)u\| \nonumber
  \end{align}
  where $u\in\mathcal{D}(\Pi(X))\cap\mathcal{D}(\Pi(Y))\cap\mathcal{D}(\Pi([X,Y]))$.
\end{theorem}

\section{Uncertainty Principles on weighted Bergman spaces}
In the following we will investigate properties of
operators arising from the discrete series representations
of the Lie group $\mathrm{SU}(1,1)$ of $2\times 2$ complex matrices with determinant 1
which leave the form $z_1\overline{w}_1-z_2\overline{w}_2$ invariant:
	\begin{align}
		\mathrm{SU}(1,1):=\left\{
		\begin{bmatrix}
			\alpha & \beta \\
			\overline{\beta} & \overline{\alpha}
		\end{bmatrix}
		:\alpha,\beta\in\mathbb{C},|{\alpha}|^2-|{\beta}|^2=1\right\}. \nonumber
	\end{align}
	The corresponding Lie algebra is
	\begin{align}
          \mathfrak{su}(1,1):=\left\{
          \begin{bmatrix}
			ia & b \\
			\overline{b} & -ia
		\end{bmatrix}
		:a\in\mathbb{R},b\in\mathbb{C}\right\}. \nonumber
	\end{align}
	For an element of $\mathrm{SU}(1,1)$,
	\begin{align}
		x=
		\begin{bmatrix}
			\alpha & \beta \\
			\overline{\beta} & \overline{\alpha}
		\end{bmatrix} \nonumber 
		\in \mathrm{SU}(1,1),
	\end{align} 
	consider the discrete series representation
        $\pi_\xi(x):\mathcal{A}_\xi^2\longrightarrow\mathcal{A}_\xi^2$ given by 
	\begin{align}
		\pi_\xi(x)f(z)
			=\dfrac{1}{(-\overline{\beta}z+\alpha)^{\xi+2}}f\left(\dfrac{\alpha z-\beta}{-\overline{\beta}z+\alpha}\right). \nonumber
	\end{align}
	It is known that $(\pi_\xi,\mathcal{A}_\xi^2)$ is a
        unitary representation of $\mathrm{SU}(1,1)$ when $\xi>-1$ is an integer.
        Moreover, it defines a unitary representation of the
        universal covering group of $\mathrm{SU}(1,1)$
        for all $\xi>-1$. In this paper we do not need to make
        a distinction between these groups, since their Lie algebras
        are the same and the exponential mapping is a local diffeomorphism.

	Let us define the following elements
        $\frak{X},\frak{Y},\frak{Z}\in\mathfrak{su}(1,1)$ by
	\begin{align}
		\frak{X}:=
			\begin{bmatrix}
				i & 0 \\
				0 & -i
			\end{bmatrix},
			~
		\frak{Y}:=
			\begin{bmatrix}
				0 & 1 \\
				1 & 0
			\end{bmatrix},
			~
		\frak{Z}:=
			\begin{bmatrix}
				i & -i \\
				i & -i
			\end{bmatrix},  \nonumber
	\end{align}
	which are a basis for $\mathfrak{su}(1,1)$.
	It follows that
	\begin{equation}
		\begin{cases}
                  \pi_\xi(\frak{X})f(z)
                  =-2izf'(z)-(\xi+2)if(z), \\
                  \pi_\xi(\frak{Y})f(z)
                  =z^2f'(z)-f'(z)+(\xi+2)zf(z),  \\
                  \pi_\xi(\frak{Z})f(z)
                  =if'(z)-2izf'(z)+iz^2f'(z)
                  +(\xi+2)izf(z)-(\xi+2)if(z).
		\end{cases} \label{operators}
	\end{equation}
	In the following, we show that an uncertainty principle on $\mathcal{A}_\xi^2$ presented in \cite{Soltani2021} can be obtained from the discrete series
        representations. 
	Define $\frak{W}\in\mathfrak{su}(1,1)$ to be
	\begin{align}
		\frak{W}:=
		\begin{bmatrix}
			0 & -i \\
			i & 0
		\end{bmatrix} = \frak{Y}-\frak{X}. \nonumber 
	\end{align}
	We see that that $\pi_\xi(\frak{W})=\pi_\xi(\frak{Z})-\pi_\xi(\frak{X})$, and since $[\frak{W},\frak{Y}]=-2\frak{X}$, then
	\begin{align}
		\pi_\xi([\frak{W},\frak{Y}])f(z)
			=-2\pi_\xi(\frak{X})f(z)
			=4izf'(z)+(\xi+2)2if(z). \label{2X}
	\end{align}

	\begin{lemma} \label{WY domain}

          The domain of $\Pi_\xi(\frak{X})$ 
          is $\mathcal{A}_{\xi,1}^2$.
	\end{lemma}
		\begin{proof} 
			We will adopt the notation that $(\xi+2)_k=\frac{\Gamma(k+\xi+2)}{\Gamma(\xi+2)}$.
			From \cite{Soltani2021}, we have for $f\in\mathcal{A}_{\xi}^2$ and $g\in\mathcal{A}_{\xi,n}^2$ that
			\begin{align}
				\|f\|_{\xi}^2&=\sum_{k=0}^\infty\abs{a_k}^2\dfrac{k!}{(\xi+2)_k} \quad\text{and}\quad
				\|g\|_{\xi,n}^2=\abs{b_0}^2+\sum_{k=1}^\infty\abs{b_k}^2\dfrac{k^{2n}k!}{(\xi+2)_k} \nonumber
			\end{align}
			where $a_k,b_k\in\mathbb{C}$. 
			Now, let $f\in\mathcal{D}(\Pi_\xi(\frak{X}))$.
			Since $f$ is holomorphic,
			\begin{align}
				\Pi_\xi(\frak{X})f(z)
				&=-2iz\dfrac{d}{dz}\sum_{k=0}^\infty a_kz^k-(\xi+2)i\sum_{k=0}^\infty a_kz^k \nonumber \\
				&=\sum_{k=0}^\infty-2ika_kz^k-\sum_{k=0}^\infty(\xi+2)ia_kz^k
					=\sum_{k=0}^\infty-i(2k+\xi+2)a_kz^k \nonumber
			\end{align}
			for $a_k\in\mathbb{C}$.
			Thus, 
			\begin{align}
				\|\Pi_\xi(\frak{X})f\|_{\xi}^2
				&=\sum_{k=0}^\infty(2k+\xi+2)^2|a_k|^2\dfrac{k!}{(\xi+2)_k}. \label{WY norm}
			\end{align}
			For any $k\geq0$ we have that
			\begin{align}
				(2k+\xi+2)^2
				&=4k^2+4k\xi+8k+\xi^2+4\xi+4 \nonumber \\
				&\leq4k^2+8k^2\xi+8k^2+4k^2\xi^2+8k^2\xi+4k^2
					=4(\xi+2)^2k^2. \nonumber 
			\end{align}
			Also, for any $k\geq1$, it follows that
			\begin{align}
				k^2
					\leq4k^2+4k\underbrace{(\xi+2)}_{>0}+\xi^2+4\underbrace{(\xi+1)}_{>0}
					=(2k+\xi+2)^2. \nonumber
			\end{align}
			Letting $c:=(\xi+2)^2|{a_0}|^2$ (the $k=0$ term from $\|\pi_\xi(X)\|_{\mathcal{A}_\xi^2}^2$), we have that
			\begin{align}
				\left(|{a_0}|^2+\sum_{k=1}^\infty |{a_k}|^2\dfrac{k^2k!}{(\xi+2)_k}\right)+(c-|{a_0}|^2) 
				&\leq\sum_{k=0}^\infty |{a_k}|^2\dfrac{(2k+\xi+2)^2k!}{(\xi+2)_k} \nonumber \\
				&\leq4(\xi+2)^2\left(|{a_0}|^2+\sum_{k=1}^\infty\abs{a_k}^2\dfrac{k^2k!}{(\xi+2)_k}\right) \nonumber
			\end{align}
			which gives the inequality
			\begin{align}
				\|f\|_{\xi,1}^2+(c-\abs{a_0}^2)
					\leq\|{\Pi_\xi(\frak{X})f}\|_{\xi}^2
					\leq4(\xi+2)^2\|{f}\|_{\xi,1}^2. \nonumber
			\end{align}
			We conclude that $\mathcal{D}(\Pi_\xi(\frak{X}))=\mathcal{A}_{\xi,1}^2$. 
		\end{proof}
	From \cite{Soltani2021}, the operators $\frac{d}{dz}$ and $z^2\frac{d}{dz}+(\xi+2)z$ have domain $\mathcal{A}_{\xi,1}^2$.
	Since $\Pi_\xi(\frak{Y})$ and $\Pi_\xi(\frak{W})$ are linear combinations of these operators, then their domains contain $\mathcal{A}_{\xi,1}^2$, i.e., $\mathcal{D}(\Pi_\xi(\frak{Y})),\mathcal{D}(\Pi_\xi(\frak{W}))\supseteq\mathcal{A}_{\xi,1}^2$.
	\begin{theorem} \label{WB UP}
          For $f\in\mathcal{A}_{\xi,1}^2$ and $w,y\in\mathbb{R}$, 
          \begin{align*}
            &(\xi+2)\|{f}\|_{\mathcal{A}_\xi^2}^2+2\ip{zf'}{f}_{\xi} \\
            &\qquad\leq\left\|{(1+z^2)f'+((\xi+2)z+iw)f)}\right\|_{\xi} 
            \left\|{(z^2-1)f'+((\xi+2)z+y)f}\right\|_{\xi}.
          \end{align*}
          Note that this uncertainty principle is identical to (2.5) in \cite{Soltani2021}.
        \end{theorem}
        \begin{proof}
          We have from Lemma \ref{WY domain} that $\mathcal{D}(\Pi_\xi([\frak{W},\frak{Y}]))=\mathcal{A}_{\xi,1}^2$. 
          Then, $\mathcal{D}(\Pi_\xi([\frak{W},\frak{Y}]))\cap\mathcal{D}(\Pi_\xi(\frak{W}))\cap\mathcal{D}(\Pi_\xi(\frak{Y}))=\mathcal{A}_{\xi,1}^2$.
          Apply Theorem \ref{Lie UP} to obtain the inequality.
        \end{proof}
        
        There are, of course, many uncertainty principles for Bergman spaces, but we only present this one here. One may ask, if there is a distinguished
        uncertainty principle.
        For example, Zhu posed the question of whether there exists an uncertainty principle on $\mathcal{A}_\xi^2$ such that the commutator is a multiple of the identity operator \cite{Zhu2015}. 
	The next result shows that it is not possible to find
        such an uncertainty principle from the representation
        $(\pi_\xi,\mathcal{A}_\xi^2)$. As it turns out (see the next section)
        this also shows that
        such an uncertainty principle cannot be found for
        self-adjoint first-order differential operators on the Bergman space.
        The reason is that any first-order differential operator
        is of the form $i\pi_\xi(X)+d$ for $X\in \mathfrak{su}(1,1)$
        and $d\in \mathbb{R}$.
        
	\begin{theorem} \label{Zhu1}
		Given $(\pi_\xi,\mathcal{A}_\xi^2)$, there are no $A,B\in\mathfrak{su}(1,1)$ such that $\Pi_\xi([A,B])=\eta$ for $\eta\in\mathbb{C}$\textbackslash$\{0\}$. 
     \end{theorem}
              \begin{proof}
                Since $\mathfrak{g}=\mathfrak{su}(1,1)$ is simple (so $[\mathfrak{g},\mathfrak{g}]=\mathfrak{g}$) and $\frak{X},\frak{Y},\frak{Z}$ form a basis, it is enough
                to show that $\Pi_\xi(\sigma \frak{X} +\tau \frak{Y} +\lambda \frak{Z})$ is
                never a non-zero multiple of the identity
                for $\sigma,\tau,\lambda\in\mathbb{R}$.
               	Notice
                \begin{align}
                  \pi_\xi(\sigma \frak{X} +\tau \frak{Y} + \lambda \frak{Z}) &=				\sigma\pi_\xi(\frak{X})+\tau\pi_\xi(\frak{Y})+\lambda\pi_\xi(\frak{Z}) \nonumber \\
                                                                             &=[(\tau+i\lambda)z^2+\overbrace{(-i2\sigma-i2\lambda)}^{*}z+(-\tau+i\lambda)]\dfrac{d}{dz} \label{skew adjoint operators} \\
				&~~~~~~~~~~~~+[(\xi+2)(\tau+i\lambda)z+\underbrace{(\xi+2)i(-\sigma-\lambda)}_{**}]. \nonumber
			\end{align}	
			To obtain a multiple of the identity,
                        the coefficient $*$ must be 0, which means $\sigma+\lambda=0$.
			However, that would mean $**$ is also 0, which tells us that we cannot obtain a non-zero multiple of the identity.
		\end{proof}
		
\section{First-order Differential Operators on weighted Bergman spaces}
In this section, we characterize self-adjoint extensions of
first-order symmetric differential
operators on $\mathcal{A}_\xi^2$. Many of the results found
here are already in \cite{Carlson2019}, but the connection
to representation theory is new. In particular we show
that \emph{all} first-order symmetric differential operators come
from the discrete series representations of $\mathrm{SU}(1,1)$.
Using this characterization, we also provide a partial answer 
to a question about uncertainty principles
raised by Yong and Zhu \cite{Zhu2015}.
Some of the operator theoretic results in this section have also been proven in more generality in \cite{Carlson2019}, but we wish to highlight the connection with representation theory and include the details here.

	Let $P(\mathbb{D})$ be the space of polynomials on $\mathbb{D}$. 
	Naturally, $P(\mathbb{D})\subseteq\mathcal{A}_\xi^2$.
	We define first-order differential operators be of the form $f\frac{d}{dz}+g$ where $f,g\in\mathcal{A}_\xi^2$.
	Denote $f(z):=\sum_{k=0}^\infty a_kz^k$ and $g(z):=\sum_{k=0}^\infty b_kz^k$ for $a_k,b_k\in\mathbb{C}$.
	Let $L:=f\frac{d}{dz}+g$ where $\mathcal{D}(L)=P(\mathbb{D})$.
	We first justify that the domain makes sense.
	\begin{lemma} \label{density}
		If $p$ is a polynomial, then $Lp$ is in $\mathcal{A}^2_\xi$ and 
		therefore the operator $L$ is densely defined.
	\end{lemma}
		\begin{proof}
			We will proceed by induction.
			When $N=0$, notice that $La_0=a_0g$, and since $g\in\mathcal{A}_\xi^2$, then $La_0\in\mathcal{A}_\xi^2$.
			Assume $L\sum_{k=0}^Mc_kz^k\in\mathcal{A}_\xi^2$ for some $M$.
			Since $L$ is a linear operator,
			\begin{align}
				L\sum_{k=0}^{M+1}c_kz^k
				&=L\left(\sum_{k=0}^Mc_kz^k+c_{M+1}z^{M+1}\right)
					=\underbrace{L\sum_{k=0}^Mc_kz^k}_{\in\mathcal{A}_\xi^2}+c_{M+1}Lz^{M+1}. \nonumber
			\end{align}
			Since $\abs{z}<1$, it follows that
			\begin{align}
				\|Lz^{M+1}\|_{\xi}
				&=\left\|\sum_{k=0}^\infty(M+1)a_kz^{k+M}+\sum_{k=0}^\infty b_kz^{k+M+1}\right\|_{\xi} \nonumber \\
				&\leq|M+1|\underbrace{\left\|\sum_{k=0}^\infty a_kz^{k+M}\right\|_{\xi}}_{\leq\|g\|_{\xi}}+\underbrace{\left\|\sum_{k=0}^\infty b_kz^{k+M+1}\right\|_{\xi}}_{\leq\|g\|_{\xi}}
					<\infty. \nonumber
			\end{align}
			Since $Lz^{M+1}\in\mathcal{A}_\xi^2$, then $L\sum_{k=0}^{M+1}c_kz^k\in\mathcal{A}_\xi^2$.
			Thus, $L$ is densely defined.
		\end{proof}
	Define the operator $T$ by $Th= fh'+gh$ and let
	$\mathcal{D}(T)$ be the set of all $h\in\mathcal{A}_\xi^2$
        for which $fh'+gh\in\mathcal{A}_\xi^2$. %, and
	Notice that
	$\mathcal{D}(L)\subseteq\mathcal{D}(T)\subseteq\mathcal{A}_\xi^2$.	
	\begin{lemma} \label{closed T}
		The operator $T$ is closed.
	\end{lemma}
	The proof follows the argument for Theorem 1.1 in \cite{Villone1970}
        or Theorem 3.9 in \cite{Carlson2019},
        but we include it here for completeness.
	The lemma implies that the operator $L$ is closable.
		\begin{proof}
			Assume that $(x_n,y_n)$ is in the graph of $T$ for all $n$ and that the sequence converges to $(x,y)$ in $\mathcal{A}_\xi^2\times \mathcal{A}_\xi^2$.
			Then $x_n$ converges to $x$ in $\mathcal{A}_\xi^2$ and therefore, by Lemma 2.4 in \cite{Zhu2000}, $x_n(z)$ and $x_n'(z)$ converge
			uniformly on compact sets to $x(z)$ and $x'(z)$ respectively. This imples that
			$y_n(z) = Tx_n(z)  = f(z)x_n'(z) + g(z)x_n(z)$ converges uniformly on compact sets to $f(z)x'(z)+g(z)x(z) = Tx(z)$. 
  			Since $y_n$ converges to $y$ in $\mathcal{A}_\xi^2$ we get that $Tx=y$ which imples that the graph of $T$ is closed.
                      \end{proof}
                       Define the space
                      \begin{equation} \label{rapiddecrease}
                        H_\xi^\infty = \left\{ f(z)\in H(\mathbb{D}) \bigg|
                      	\sum_{k=0}^\infty |a_k|^2\| z^k\|_{\xi}^2 (\xi(\xi+2) +2k)^{2m}<\infty,m\in\mathbb{N}\right\}.
                    \end{equation}
                    This is the space of smooth vectors for the representation
                    $\pi_\xi$ \cite{Olafsson1988} and more generally
                    \cite{Chebli2004}. Notice that if $f\in H_\xi^\infty$, then
                      $f'$ is also in $H_\xi^\infty$, since
                      $$
                      \| f'\|_{\xi}^2
                      = \sum_{k=1}^\infty |ka_k|\| z^{k-1}\|_{\xi}^2
                      $$
                      and $\| z^{k-1}\|_{\xi}^2
                      = \frac{k + \xi +1}{k}\| z^k\|_{\xi}^2$ (choose $m=1$ in (\ref{rapiddecrease}) and compare norms).

                      \begin{lemma} \label{equal closure}
                        If $f$ and $g$ are polynomials, then
                        $S = f(z)\frac{d}{dz} + g(z)$ with domain
                        $\mathcal{D}(S)=H_\xi^\infty$ is densely defined.
                        Moreover, $S\subseteq T$ and is therefore closable.
                        Lastly, $\overline{L}=\overline{S}$.
                      \end{lemma}
                      Notice that the space of smooth vectors is
                      related (but not equal) to the spaces
                      introduced in section 3.1 of \cite{Carlson2019}.
                      
                      \begin{proof}
                        We only need to prove the last statement, which
                        requires that the graph $G(L)$ is dense in
                        the graph $G(\overline{S})=\overline{G(S)}$.
                        Since multiplication by $g(z)$
                        is a bounded operator on $\mathcal{A}_\xi^2$, it
                        is enough to work with $g=0$.
                        Let $f(z) = \sum_{k=0}^mb_kz^k$.
                        If $(h,\overline{S}h)$ is 
                        in $\overline{G(S)}$, then there
                        is a vector $\widetilde{h}\in H_\xi^\infty$ such that
                        $\| h-\widetilde{h}\|_{\xi}$ and $\| \overline{S}h-S\widetilde{h}\|_{\xi}$
                        are both small. Let 
                        $\widetilde{h}(z) = \sum_{k=0}^\infty a_kz^k$
                        with coefficient satisfying (\ref{rapiddecrease}).
                        Define the polynomial
                        $\widetilde{h}_n(z) = \sum_{k=0}^n a_kz^k$. 
                        Then
                        $
                        S\widetilde{h}-L\widetilde{h}_n
                        = S(\widetilde{h}-\widetilde{h}_n)$
                        has norm
                        \begin{equation} \label{normofSh}
                        \| S\widetilde{h}-L\widetilde{h}_n\|_{\xi}
                        \leq \sum_{\ell=0}^m |b_\ell|
                        \left( \sum_{k=n+1}^\infty |ka_k|^2 \|z^{k+\ell-1}\|_{\xi}^2 \right)^{1/2}
                      \end{equation}
                      Since
                        $\| z^{k-1}\|_{\xi}^2
                        = \frac{\xi +1 +k}{k}\| z^k\|_{\xi}^2 \leq
                        (\xi+2) \| z^k\|_{\xi}^2 $
                        and
                        $
                        \frac{\| z^{k+\ell}\|_{\xi}^2}{\| z^k\|_{\xi}^2} = \frac{(k+\ell)!\Gamma(k+\xi+2)}{k!\Gamma(k+\ell+\xi+2)}$
                        converges to 1 as $k\to\infty$, by the ratio test and (\ref{rapiddecrease}), the finite sum in (\ref{normofSh})
                        converges to $0$ as $n\to\infty$.
                        This shows that $\| h-\widetilde{h}_n\|_{\xi}$
                        and $\| \overline{S}h-L\widetilde{h}_n\|_{\xi}$ are both small when
$n$ is large enough, and therefore $G(L)$ is dense in $G(\overline{S})$.                     
                      \end{proof}

                      \begin{theorem} \label{order}
		The operator $L=f\frac{d}{dz}+g$ with domain $\mathcal{D}(L)=P(\mathbb{D})$ 
		is symmetric if and only if $f=a_0+a_1z+\overline{a_0}z^2$ and $g=b_0+(\xi+2)\overline{a_0}z$, where $a_1,b_0\in\mathbb{R}$.
	\end{theorem}
			The first part of the proof is similar to \cite{Villone1970}, but is included here for completeness.
		\begin{proof}
			Since any polynomial is a finite linear combination of the $e_n$'s, it is enough to work with the basis elements.
			
			Assume $L$ is symmetric. For any $z\in\mathbb{D}$,
			\begin{align}
				Le_n(z)
				&=\sum_{k=0}^\infty na_k\sqrt{\frac{(\xi+2)_n}{n!}}z^{k+n-1}+\sum_{k=0}^\infty b_k\sqrt{\frac{(\xi+2)_n}{n!}}z^{k+n} \nonumber \\
				&=\sum_{k=0}^\infty na_{k-n+1}\sqrt{\frac{(\xi+2)_n}{n!}}z^{k}+\sum_{k=0}^\infty b_{k-n}\sqrt{\frac{(\xi+2)_n}{n!}}z^{k} \nonumber \\
				&=\sum_{k=0}^\infty na_{k-n+1}\sqrt{\frac{(\xi+2)_nk!}{(\xi+2)_kn!}}e_k(z)+\sum_{k=0}^\infty b_{k-n}\sqrt{\frac{(\xi+2)_nk!}{(\xi+2)_kn!}}e_k(z). \nonumber \\
				&=\sum_{k=0}^\infty\sqrt{\frac{(\xi+2)_nk!}{(\xi+2)_kn!}}(na_{k-n+1}+b_{k-n})e_k(z). \nonumber
			\end{align}
			Note that $e_k(z)$ is a multiple of $z^k$ and $\frac{k!}{(\xi+2)_k}\leq1$.
			By the series comparison test using the results of Lemma 2.1 in \cite{Villone1970}, the series $Le_n(z)$ converges, so
			\begin{align}
				Le_n
				&=\sum_{k=0}^\infty\sqrt{\frac{(\xi+2)_nk!}{(\xi+2)_kn!}}(na_{k-n+1}+b_{k-n})e_k. \nonumber
			\end{align}
			Now, by orthonormality of the $\{e_k\}_{k\in\mathbb{N}_0}$, we have that
			\begin{align}
				\ip{Le_n}{e_m}_{\xi}
				&=\sqrt{\frac{(\xi+2)_nm!}{(\xi+2)_mn!}}(na_{m-n+1}+b_{m-n}), \nonumber
			\end{align}
			and similarly for $\ip{e_n}{Le_m}_{\xi}$,
			\begin{align}
				\ip{e_n}{Le_m}_{\xi}
				&=\sqrt{\frac{(\xi+2)_mn!}{(\xi+2)_nm!}}(m\overline{a_{n-m+1}}+\overline{b_{n-m}}). \nonumber
			\end{align}
			Notice that $L$ is symmetric if and only if
			\begin{align}
				\sqrt{\frac{(\xi+2)_nm!}{(\xi+2)_mn!}}(na_{m-n+1}+b_{m-n})
				&=\sqrt{\frac{(\xi+2)_mn!}{(\xi+2)_nm!}}(m\overline{a_{n-m+1}}+\overline{b_{n-m}}) \label{formal sym}
			\end{align}
			for $n,m\in\mathbb{N}_0$.
			By properties of the Gamma function, 
			\begin{align}
				(\xi+2)_1
				&=\frac{\Gamma(\xi+2+1)}{\Gamma(\xi+2)}
					=\frac{(\xi+2)\Gamma(\xi+2)}{\Gamma(\xi+2)}
					=(\xi+2) \nonumber,
			\end{align}
			so we have in particular that
			\begin{align}
				g
				&=Le_0
					=\sum_{k=0}^\infty\sqrt{\frac{k!}{(\xi+2)_k}}b_{k}e_k
					=\sum_{k=0}^\infty\sqrt{\frac{(\xi+2)_k}{k!}}(k\overline{a_{1-k}}+\overline{b_{-k}})e_k \nonumber \\
				&=\overline{b_0}+\sqrt{(\xi+2)_1}\overline{a_0}e_1
					=\overline{b_0}+(\xi+2)_1\overline{a_0}z 
					=\overline{b_0}+(\xi+2)\overline{a_0}z. \nonumber
			\end{align}
			Using (\ref{formal sym}) for $m=0$ and $n=0$, we have $b_0=\overline{b_0}\in\mathbb{R}$.
			Then, we derive that
			\begin{align}
				f
				&=\frac{1}{\sqrt{\xi+2}}Le_1-gz \nonumber \\
				&=\left(\frac{1}{\sqrt{\xi+2}}\sum_{k=0}^\infty\sqrt{\frac{(\xi+2)k!}{(\xi+2)_k}}(a_k+b_{k-1})e_k\right)-(\overline{b_0}+(\xi+2)\overline{a_0}z)z \nonumber \\
				&=\left(\frac{1}{\xi+2}\sum_{k=0}^\infty\sqrt{\frac{(\xi+2)_k}{k!}}(k\overline{a_{2-k}}+\overline{b_{1-k}})e_k\right)-(\overline{b_0}z+(\xi+2)\overline{a_0}z^2) \nonumber \\
				&=\left[\frac{\overline{b_1}}{\xi+2}+(\overline{a_1}+\overline{b_0})z+\frac{(\xi+2)_2}{\xi+2}\overline{a_0}z^2\right]-(\overline{b_0}z+(\xi+2)\overline{a_0}z^2) \nonumber \\
				&=\frac{\overline{b_1}}{\xi+2}+\overline{a_1}z+\underbrace{\left(\frac{(\xi+2)_2}{\xi+2}-(\xi+2)\right)}_{=1}\overline{a_0}z^2
					=\frac{\overline{b_1}}{\xi+2}+\overline{a_1}z+\overline{\alpha_0}z^2. \nonumber 
			\end{align}
			Using (\ref{formal sym}) for $m=0$ and $n=1$, 
			\begin{align}
				\sqrt{(\xi+2)}a_0=\frac{\overline{b_1}}{\sqrt{\xi+2}}
				\iff a_0=\frac{\overline{b_1}}{\xi+2}, \nonumber
			\end{align}
			and for $m=1$ and $n=1$, we get that $a_1=\overline{a_1}\in\mathbb{R}$.	
				
			Now assume that $f,g$ are of the desribed form, then $L=(a_0+a_1z+\overline{a_0}z^2)\frac{d}{dz}+(b_0+(\xi+2)\overline{a_0}z)$ for $a_1,b_0\in\mathbb{R}$.
			Adopting the convention that $e_{-k}=0$ for $k\in\mathbb{N}$, we have
			\begin{align}
				&Le_n
					=(a_0+a_1z+\overline{a_0}z^2)\frac{d}{dz}+(b_0+(\xi+2)\overline{a_0}z)]\sqrt{\frac{(\xi+2)_n}{n!}}z^n \nonumber \\
				&=[na_0z^{n-1}+(na_1+b_0)z^n+(n+\xi+2)\overline{a_0}z^{n+1}]\sqrt{\frac{(\xi+2)_n}{n!}}\nonumber \\
				&=na_0\sqrt{\frac{(\xi+2)_n}{n(\xi+2)_{n-1}}}e_{n-1}+(na_1+b_0)e_n+(n+\xi+2)\overline{a_0}\sqrt{\frac{(n+1)(\xi+2)_n}{(\xi+2)_{n+1}}}e_{n+1} \nonumber \\
				&=a_0\sqrt{n(n+\xi+1)}e_{n-1}+(na_1+b_0)e_n+\overline{a_0}\sqrt{(n+\xi+2)(n+1)}e_{n+1}. \nonumber
			\end{align}
			It follows by orthogonality that
			\begin{align}
				\ip{Le_n}{e_{n-1}}_\xi
				&=a_0\sqrt{n(n+\xi+1)} 
					=\ip{e_n}{Le_{n-1}}_\xi, \nonumber \\
				\ip{Le_n}{e_n}_\xi
				&=na_1+b_0
					=\ip{e_n}{Le_n}_\xi \nonumber \\
				\ip{Le_n}{e_{n+1}}_\xi
				&=\overline{a_0}\sqrt{(n+\xi+2)(n+1)}
					=\ip{e_n}{Le_{n+1}}_\xi, \nonumber
			\end{align}
			and that $\ip{Le_n}{e_m}_\xi=0=\ip{e_n}{Le_m}_\xi$ when $m$ is not $n-1$, $n$ or $n+1$, so $L$ is symmetric.
                      \end{proof}

                      One of our main results in this section is the following
                      classification of self-adjoint first-order
                      differential operators in terms of the discrete series
                      representations.
                \begin{theorem} \label{self adjoint operator}
		Given $a,b\in\mathbb{R}$ and $c\in\mathbb{C}$, define $L=(cz^2+az+\overline{c})\dfrac{d}{dz}+((\xi+2)cz+b)$ with $\mathcal{D}(L)=P(\mathbb{D})$.
		Then $L$ is symmetric and equal to the restriction of an operator of the form $S=i\pi_\xi(X) + d$ with $D(S)=H_\xi^\infty$ to 
		for $d$ real and $X\in\mathfrak{su}(1,1)$.
		Moreover, $\overline{L} = i\Pi_\xi(X)+d$ is self-adjoint.
	\end{theorem}
		\begin{proof}
			Theorem \ref{order} gives that $L$ is symmetric.
			Let $\sigma,\tau,\lambda\in\mathbb{R}$.
			For any $X\in\mathfrak{su}(1,1)$, we can express $i\pi_\xi(X)$ as
			\begin{align}
			 	i(\sigma\pi_\xi(\frak{X})+\tau\pi_\xi(\frak{Y})+\lambda\pi_\xi(\frak{Z}))
			 	&=[(-\lambda+i\tau)z^2+(2\sigma+2\lambda)z+(-\lambda-i\tau)]\dfrac{d}{dz} \nonumber \\
			 	&~~~~~~~~+[(\xi+2)(-\lambda+i\tau)z+(\xi+2)(\sigma+\lambda)]. \nonumber
			\end{align}
			which can be simplified to
			\begin{align}
				[(-\lambda+i\tau)z^2+(2\sigma+2\lambda)z+(-\lambda-i\tau)]\dfrac{d}{dz}+[(\xi+2)(-\lambda+i\tau)z+2(\sigma+\lambda)]. \nonumber
			\end{align}
			This is equivalent to saying 
			\begin{align}
				i\pi_\xi(X)
					=(cz^2+az+\overline{c})\dfrac{d}{dz}+((\xi+2)cz+a) \nonumber
			\end{align}
			for $a\in\mathbb{R}$ and $c\in\mathbb{C}$.
			Notice that
			\begin{align}
				L
				&=\overbrace{(cz^2+az+\overline{c})\dfrac{d}{dz}+((\xi+2)cz+b)}^{\text{F-O DO on $\mathcal{A}_\xi^2$}} \nonumber \\
				&=\underbrace{(cz^2+az+\overline{c})\dfrac{d}{dz}+((\xi+2)cz+b)+(a-b)}_{\text{F-O DO from $(\pi_\xi,\mathcal{A}_\xi^2)$}}
					=i\pi_\xi(X)+d. \nonumber
			\end{align}
			We know that $\overline{L}=\overline{S}$ from Lemma \ref{equal closure}.
                        That $\overline{S}=i\Pi(X)+d$
                        is self-adjoint has been shown in 
                        \cite{Segal1951}.
		\end{proof}
		
		By a similar argument to Villone \cite[Theorem 1.2]{Villone1970}, this implies that $\overline{L}=T$.

                \begin{remark}
                  This results shows that the polynomials are a core for
                  $\Pi_\xi(X)$ for every $X\in\mathfrak{su}(1,1)$ for
                  the discrete series representation. This seems to be
                  a new result, although it is likely not surprising
                  to experts.
                  
                  It seems that one should be able to characterize
                  self-adjoint differential operators of any order,
                  studied in \cite{Villone1970},
                  via the
                  extension of $\pi_\xi$ to the universal enveloping algebra of
                  $\mathfrak{su}(1,1)$.
                \end{remark}

                Theorem~\ref{self adjoint operator} shows that any self-adjoint first-order differential operator on $\mathcal{A}_\xi^2$ is generated by $\Pi_\xi$ and the identity. 
		This allows us to strengthen our answer to Zhu's question in \cite{Zhu2015}.
		\begin{theorem} \label{Zhu3}
			There are no first-order self-adjoint differential operators $A,B$ on $\mathcal{A}_\xi^2$ such that $\overline{[A,B]}=\eta$ for $\eta\in\mathbb{C}$\textbackslash$\{0\}$.  
		\end{theorem}
			\begin{proof}
				Let $A,B$ be first-order self-adjoint 
differential operators on $\mathcal{A}_\xi^2$. We know that $A=\Pi_\xi(X)$ and
$B=\Pi_\xi(Y)$ for some $X,Y\in\mathfrak{su}(1,1)$.
Then
\begin{align}
  [A,B] = 
  [\Pi_\xi(X),\Pi_\xi(Y)]
                                  \subseteq \Pi_\xi([X,Y]). \nonumber
				\end{align}
				Apply the fact that $\overline{[A,B]}=\Pi_\xi([X,Y])$ which can never be a non-zero 
multiple of the identity by Theorem \ref{Zhu1}.
\end{proof}

\section{Relation between Bergman spaces with shifted weights}
In this section, we show that the weighted Bergman space, $\mathcal{A}_\xi^2$, can be related to a shifted weighted Bergman space.

	\begin{theorem} \label{isomorphism}
		The operator $\widetilde{L}:=z\frac{d}{dz}+c$ with $c$ not being 0 or a negative integer extends to an isomorphism between $\mathcal{A}_\xi^2$ and $\mathcal{A}_{\xi+2}^2$.
	\end{theorem} 
		\begin{proof}
			The argument for linearity is clear.
			Let $f,g\in\mathcal{A}_\xi^2$.
			Denote $f(z):=\sum_{k=0}^\infty a_kz^k$, and let $g(z):=\sum_{k=0}^\infty b_kz^k$.
			Then
			\begin{align}
				\widetilde{L}f(z)=\widetilde{L}g(z)
				&\iff\sum_{k=0}^\infty(k+c)a_kz^k=\sum_{k=0}^\infty(k+c)b_kz^k 
					\iff a_k=b_k \nonumber
			\end{align}
			for every $k\in\mathbb{N}_0$, so $\widetilde{L}$ is injective.
			
			Let $f\in\mathcal{A}_{\xi+2}^2$ and denote $f(z):=\sum_{k=0}^\infty a_kz^k$, and let $g(z):=\sum_{k=0}^\infty\frac{a_k}{k+c}z^k$.
			Notice that $\frac{(k+c)^2}{(k+\xi+3)(k+\xi+2)}\to1$ as $k\to\infty$, and is always positive if $c\not\in \{\dots,-3,-2,-1,0 \}$, so there is $D>0$, such that $\frac{1}{D}\leq\frac{(k+c)^2}{(k+\xi+3)(k+\xi+2)}\leq D$ for any $k\in\mathbb{N}_0$.
			Then,
			\begin{align}
				\|g\|_\xi^2
				&=\sum_{k=0}^\infty\frac{1}{(k+c)^2}|a_k|^2\|z^k\|_{\xi}^2 \nonumber \\
				&\leq\frac{D}{(\xi+3)(\xi+2)}\sum_{k=0}^\infty\frac{(\xi+3)(\xi+2)}{(k+\xi+3)(k+\xi+2)}|a_k|^2\|z^k\|_\xi^2 \nonumber \\
				&=\frac{D}{(\xi+3)(\xi+2)}\sum_{k=0}^\infty|a_k|^2\|z^k\|_{\xi+2}^2
					=\frac{D}{(\xi+3)(\xi+2)}\|f\|_{\xi+2}^2
					<\infty, \nonumber
			\end{align}
			so $g\in\mathcal{A}_\xi^2$.
			Since $\widetilde{L}g(z)=f(z)$ for every $z\in\mathbb{D}$, then $\widetilde{L}$ is surjective.
			For any $f\in\mathcal{A}_\xi^2$ denoted $f(z):=\sum_{k=0}^\infty a_kz^k$, we have
			\begin{align}
				&\frac{(\xi+3)(\xi+2)}{D}\underbrace{\sum_{k=0}^\infty|a_k|^2\|z^k\|_\xi^2}_{=\|f\|_\xi^2}
					=(\xi+3)(\xi+2)\sum_{k=0}^\infty\frac{1}{D}|a_k|^2\|z^k\|_\xi^2 \nonumber \\
				&\leq\underbrace{(\xi+3)(\xi+2)\sum_{k=0}^\infty\frac{(k+c)^2}{(k+\xi+3)(k+\xi+2)}|a_k|^2\|z^k\|_\xi^2}_{=\|\widetilde{L}f\|_{\xi+2}^2} \nonumber \\
				&\leq (\xi+3)(\xi+2)\sum_{k=0}^\infty D|a_k|^2\|z^k\|_\xi^2
					=D(\xi+3)(\xi+2)\underbrace{\sum_{k=0}^\infty|a_k|^2\|z^k\|_\xi^2}_{=\|f\|_\xi^2}, \nonumber
			\end{align}
			so $\widetilde{L}$ is continuous.
		\end{proof}
			
       \begin{remark} \label{xi+2}
                  
		If $\widetilde{L}=\frac{2}{\xi+2}z\frac{d}{dz}+1$ and $K_\xi^w(z)=K_\xi(z,w)$ is the reproducing kernel
    		for $\mathcal{A}^2_\xi$, then
          $\widetilde{L}K_\xi^w(z)
          = K_{\xi+1}^w(z) = K_{\xi+1}(z,w)$. Such differential operators
          have been considered in \cite{Yan1992}.
          Also, Theorem \ref{isomorphism} shows that $\mathcal{D}(\Pi_\xi(\frak{X}))=\mathcal{A}_{\xi-2}^2=\mathcal{A}_{\xi,1}^2$ as sets when $\xi>1$.
	\end{remark}

	Zhu in \cite[p.~50]{Zhu2000} shows that $\mathcal{R}:=z\frac{d}{dz}$ is a surjective map from $\mathcal{A}_\xi$ to $\mathcal{A}_{\xi+2}^2$.
	Theorem \ref{isomorphism} shows that $\mathcal{R}$ (surjective) plus particular $c's$ (injective) is an isomorphism between the spaces.
	This is not necessarily true for any $c\in\mathbb{R}$.
	We now state a generalization we believe is possible.
	\begin{remark} \label{conjecture}
          We conjecture that Theorem \ref{isomorphism} can be extended to show that $\widetilde{L}$ extends to an isomorphism between $\mathcal{A}_\xi^p$ and $\mathcal{A}_{\xi+p}^p$. This would involve a study of multipliers
          between the two spaces.
          This also opens up the question of which
          differential operators $q_1\frac{d}{dz}+q_2$
          extend to isomorphisms between Bergman spaces if $q_1,q_2$ are polynomials.

	\end{remark}

\bibliographystyle{plain}

\end{document}